\documentclass[11pt]{amsart}
\usepackage{amssymb}
\usepackage{amscd}
\usepackage[all]{xy}
\usepackage[pagebackref,colorlinks,linkcolor=red,citecolor=blue,urlcolor=blue,hypertexnames=true]{hyperref}

\numberwithin{equation}{section}

\def\today{\number\day\space\ifcase\month\or   January\or February\or
   March\or April\or May\or June\or   July\or August\or September\or
   October\or November\or December\fi\   \number\year}

\newcounter{TmpEnumi}

\theoremstyle{definition}
\newtheorem{thm}{Theorem}[section]
\newtheorem{lem}[thm]{Lemma}

\newtheorem{cor}[thm]{Corollary}

\newtheorem{rmk}[thm]{Remark}

\newtheorem{qst}[thm]{Question}

\newcommand{\beq}{\begin{equation}}
\newcommand{\eeq}{\end{equation}}
\newcommand{\beqr}{\begin{eqnarray*}}
\newcommand{\eeqr}{\end{eqnarray*}}
\newcommand{\bal}{\begin{align*}}
\newcommand{\eal}{\end{align*}}
\newcommand{\bei}{\begin{itemize}}
\newcommand{\eei}{\end{itemize}}
\newcommand{\limi}[1]{\lim_{{#1} \to \infty}}

\newcommand{\af}{\alpha}
\newcommand{\bt}{\beta}
\newcommand{\gm}{\gamma}
\newcommand{\dt}{\delta}
\newcommand{\ep}{\varepsilon}
\newcommand{\zt}{\zeta}
\newcommand{\et}{\eta}

\newcommand{\te}{\theta}
\newcommand{\ld}{\lambda}
\newcommand{\sm}{\sigma}
\newcommand{\kp}{\kappa}
\newcommand{\ph}{\varphi}
\newcommand{\ps}{\psi}
\newcommand{\rh}{\rho}
\newcommand{\om}{\omega}
\newcommand{\ta}{\tau}

\newcommand{\Gm}{\Gamma}

\newcommand{\Q}{{\mathbb{Q}}}
\newcommand{\Z}{{\mathbb{Z}}}
\newcommand{\R}{{\mathbb{R}}}
\newcommand{\C}{{\mathbb{C}}}
\newcommand{\N}{{\mathbb{Z}}_{> 0}}
\newcommand{\Nz}{{\mathbb{Z}}_{\geq 0}}
\newcommand{\T}{{\mathbb{T}}}

\newcommand{\Tr}{{\mathrm{Tr}}}
\newcommand{\tr}{{\mathrm{tr}}}
\newcommand{\id}{{\mathrm{id}}}
\newcommand{\ev}{{\mathrm{ev}}}

\newcommand{\sa}{{\mathrm{sa}}}

\newcommand{\Prim}{{\mathrm{Prim}}}
\newcommand{\diag}{{\mathrm{diag}}}

\newcommand{\card}{{\mathrm{card}}}
\newcommand{\Aut}{{\mathrm{Aut}}}

\newcommand{\andeqn}{\quad
{\mbox{and}} \quad}

\newcommand{\Wolog}{Without loss of generality}

\newcommand{\ca}{C*-algebra}

\newcommand{\hm}{homomorphism}

\newcommand{\ct}{continuous}
\newcommand{\cfn}{continuous function}

\newcommand{\cms}{compact metric space}

\newcommand{\mh}{minimal homeomorphism}

\newcommand{\op}{{\mathrm{op}}}

\newcommand{\ov}{\overline}

\newcommand{\I}{\infty}

\title[Not equivariantly isomorphic to its opposite]{Simple
 nuclear C*-algebras not equivariantly isomorphic to their opposites}

\author{Marius Dadarlat}
\address{Department of Mathematics,
    Purdue University,
    150 N.\  University Street,
    West Lafayette IN 47907-2067,
USA.}

\author{Ilan Hirshberg}
\address{Department of Mathematics,
Ben-Gurion University of the Negev,
Be'er Sheva 8410501,
Israel.}

\author{N.~Christopher Phillips}
\address{Department of Mathematics,
University of Oregon,
Eugene OR 97403-1222,
USA.}

\date{14~February 2016}

\thanks{This research was supported in part
by the US-Israel Binational Science
Foundation.
This material is based upon work supported by the
  US National Science Foundation under
  Grant DMS-1101742 and  Grant DMS-1362824.}

\begin{document}

\begin{abstract}
We exhibit examples of simple separable nuclear \ca{s},
along with actions of the circle group
and outer actions of the integers,
which are not equivariantly isomorphic to their opposite algebras.
In fact,
the fixed point subalgebras are not isomorphic to their opposites.
The \ca{s} we exhibit are well behaved
from the perspective of structure and classification
of nuclear \ca{s}:
they are unital \ca{s} in the UCT class,
with finite nuclear dimension.
One is an AH-algebra with
unique tracial state
and absorbs the CAR algebra tensorially.
The other is a Kirchberg algebra.
\end{abstract}

\maketitle

Let $A$ be a \ca.
We denote by $A^{\op}$ the opposite algebra:
the same Banach space with the same involution,
but with reversed multiplication.
The question of constructing operator algebras
not isomorphic to their opposites goes back to~\cite{connes},
which constructs factors not isomorphic to their opposites.
Separable simple \ca{s} not isomorphic to their opposites
were constructed in \cite{phillips-PAMS,phillips-viola},
but those examples are not nuclear.
In the nuclear setting,
there are nonsimple \ca{s}
not isomorphic to their opposites
(\cite{PhX}; see also \cite{rosenberg} for a related discussion).
The question of whether there are simple nuclear \ca{s}
not isomorphic to their opposites
is an important and difficult open question,
particularly due to its connection
with the Elliott classification program.
The Elliott invariant,
as well as the Cuntz semigroup,
cannot distinguish a \ca{} from its opposite.
Thus,
existence of a simple separable
nuclear \ca{} not isomorphic to its opposite
would reveal an entirely new phenomenon.

In this paper,
we address the equivariant situation.
We exhibit examples of simple separable unital nuclear \ca{s}~$A$
along with outer actions of~$\Z$
and with actions of the circle group~$\T$
which are not \emph{equivariantly} isomorphic to their opposites.
The \ca{s}~$A$ are well behaved
from the perspective of structure and classification of \ca{s}.
In one set of examples,
$A$ is AH with no dimension growth,
has a unique tracial state, and tensorially absorbs the CAR algebra.
In the other, $A$ is a Kirchberg algebra
satisfying the Universal Coefficient Theorem.

In fact, in our examples
the fixed point algebras
are not isomorphic to their opposite algebras.
In particular,
we give outer actions of $\Z$ on a simple separable unital nuclear
\ca{} with tracial rank zero
and on a unital Kirchberg algebra,
both satisfying the Universal Coefficient Theorem,
such that the fixed point algebras
are not isomorphic to their opposites.

These examples illustrate some of the difficulties one would encounter
if one wished to extend the current classification results
to the equivariant setting,
for actions of both $\Z$ and~$\T$.

In the rest of the introduction,
we recall a few general facts about opposite algebras.
Section~\ref{Sec_Coh} contains some preparatory lemmas,
and Section~\ref{Sec_Ex} contains the construction of our examples.
In Section~\ref{Sec_Diff}
we collect several remarks on our construction,
outline a shorter construction which gives examples with some
of the properties of our main examples,
and state some open questions.

If $A$ is a \ca,
we denote by $A^{\#}$ its conjugate algebra.
As a real \ca{} it is the same as~$A$,
but it has the reverse complex structure.
That is,
if we denote its scalar multiplication by
$(\ld, a) \mapsto \ld \bullet_{\#} a$,
then $\lambda \bullet_{\#} a = \overline{\lambda}a$
for any $a \in A$ and any $\lambda \in \C$.
We recall the following easy fact.

\begin{lem}\label{L_5Y11_Conj}
Let $A$ be a \ca.
Then the map $a \mapsto a^*$
is an isomorphism $A^{\op} \to A^{\#}$.
\end{lem}

In this paper,
we often find it more convenient to use~$A^{\#}$.

Let $G$ be a locally compact Hausdorff group,
and let $\alpha \colon G \to \Aut (A)$
be a point-norm continuous action.
For $g \in G$,
the same map $\alpha_g$,
viewed as a map from $A^{\#}$ to itself,
is also an automorphism.
To see that,
we note that it is clearly a real \ca{} automorphism,
and for each $\lambda \in \C$ and $a \in A$,
we have
\[
\alpha_g (\lambda \bullet_{\#} a)
 = \alpha_g \big( \overline{\lambda} a \big)
 = \overline{\lambda} \alpha_g (a)
 = \lambda \bullet_{\#} \alpha_g (a).
\]
Thus,
the same map gives us an action $\af^{\#}$ of $G$ on $A^{\#}$,
which we call the {\emph{conjugate action}}.
The definition of the crossed product shows that
$(A \rtimes_{\alpha} G)^{\#} \cong A^{\#} \rtimes_{\alpha^{\#}} G$:
they are identical as real \ca{s},
and the complex structure on $A \rtimes_{\alpha} G$
comes from the complex structure on~$A$.
The same map also gives an action
$\af^{\op}$ of $G$ on $A^{\op}$,
which we call the {\emph{opposite action}}.
The map from Lemma~\ref{L_5Y11_Conj}
intertwines $\af^{\op}$ with $\af^{\#}$ and hence
$(G, A^{\op}, \alpha^{\op})\cong (G, A^{\#}, \alpha^{\#})$.
The identification of the crossed product, however, is less direct.

If $(G, A, \alpha)$ and $(G, B, \beta)$ are $G$-\ca{s}
and are $G$-equivariantly isomorphic,
then $A \rtimes_{\alpha} G \cong B \rtimes_{\beta} G$.
Thus,
an equivariant version
of the problem of whether \ca{s} are isomorphic
to their opposites is whether $(G, A, \alpha)$
is isomorphic to $(G, A^{\#}, \alpha^{\#})$.

\section{Preparatory lemmas}\label{Sec_Coh}

Our construction requires two lemmas from cohomology,
some properties of a particular finite group,
a result on quasidiagonality of crossed products
of integer actions on section algebras of \ct{} fields,
and a lemma concerning tracial states on crossed products
by an automorphism with finite Rokhlin dimension.

The following lemma generalizes Lemma~3.6 of~\cite{PhX},
which is part of an example
of a topological space with specific properties
originally suggested by Greg Kuperberg.

\begin{lem}\label{L_4909_SameSig}
Let $n \in \N$,
let $M$ be a connected compact orientable manifold of dimension~$4 n$
and with no boundary,
and let $h \colon M \to M$ be a \cfn{}
such that $h_* \colon H_* (M; \Z) \to H_* (M; \Z)$
is an isomorphism.
Suppose that the signature of $M$ is nonzero.
Then $h$ is orientation preserving.
\end{lem}

\begin{proof}
We recall the definition of the signature,
starting with the bilinear form
$\om \colon H^{2 n} (M; \Z) \times H^{2 n} (M; \Z) \to \Z$
defined as follows.
Let $e_0 \in H_0 (M; \Z)$
be the standard generator.
Thus there is an isomorphism $\nu \colon H_0 (M; \Z) \to \Z$
such that $\nu (k e_0) = k$ for all $k \in \Z$.
Further let $c \in H_{4 n} (M; \Z)$
be the generator corresponding to the orientation of~$M$
(the fundamental class).
Also recall the cup product
$(\af, \bt) \mapsto \af \smile \bt$
from $H^k (M; \Z) \times H^l (M; \Z)$ to $H^{k + l} (M; \Z)$
and the cap product
$(\af, \bt) \mapsto \af \frown \bt$
from $H^k (M; \Z) \times H_l (M; \Z)$ to $H_{l - k} (M; \Z)$.
Then $\om$ is given by
\[
\om (\af, \bt)
 = \nu ( [\af \smile \bt] \frown c )
\]
for $\af, \bt \in H^{2 n} (M; \Z)$.
The signature of the form gotten by tensoring with $\R$ is,
by definition,
the signature of~$M$.

The Universal Coefficient Theorem
(in~\cite{hatcher} see Theorem~3.2 and page 198)
and the Five Lemma imply that
$h^* \colon H^* (M; \Z) \to H^* (M; \Z)$
is an isomorphism.
In particular,
$h^* \colon H^{2 n} (M; \Z) \to H^{2 n} (M; \Z)$
is an isomorphism.
Therefore the bilinear form $\rh$ on $H^{2 n} (M; \Z)$,
given by
\[
\rh (\af, \bt)
 = \om \big( h^* (\af), \, h^* (\bt) \big)
 = \nu \big( [h^* (\af) \smile h^* (\bt)] \frown c \big)
\]
for $\af, \bt \in H^{2 n} (M; \Z)$,
is equivalent to~$\om$.
In particular,
$\rh$ has the same signature as~$\om$.

Now define a bilinear form $\om_0$ on $H^{2 n} (M; \Z)$
by
\[
\om_0 (\af, \bt)
 = \nu \big( [h^* (\af) \smile h^* (\bt)] \frown (h_*)^{-1} (c) \big)
\]
for $\af, \bt \in H^{2 n} (M; \Z)$.
The formula for $\om_0$ differs
from the formula for $\rh$ only in that
$c$ has been replaced by $(h_*)^{-1} (c)$.
The maps $\nu \circ h_*$ and $\nu$
agree on $H_0 (M; \Z)$.
(This is true for any \ct{} map $h \colon M \to M$.)
Naturality of the cup and cap products
therefore implies that $\om = \om_0$.
If $(h_*)^{-1} (c) = - c$,
then $\om_0 = - \rh$,
so $\om_0$ and $\rh$ have opposite signatures.
Since $\om_0 = \om$ and $\rh$ have the same signature
by the previous paragraph,
we find that the signature of $\om$ is zero.
This contradiction shows that $(h_*)^{-1} (c) \neq - c$.

Since $h_*$ is an isomorphism and $H_{4 n} (M; \Z) \cong \Z$,
it follows that $h_* (c) = \pm c$.
The previous paragraph rules out $h_* (c) = - c$,
so $h_* (c) = c$.
\end{proof}

\begin{lem}\label{L_4909_IsoAfterSn}
Let $m \in \N$ and
let $M$ be a connected compact orientable manifold of dimension~$m$.
Let $n \in \N$ satisfy $n > m$,
and let $h \colon S^n \times M \to S^n \times M$
be a \cfn{}.
Let $y_0 \in S^n$,
let $i \colon M \to S^n \times M$
be $i (x) = (y_0, x)$ for $x \in M$,
and let $p \colon S^n \times M \to M$
be the projection on the second factor.
Then:
\begin{enumerate}
\item\label{L_4909_IsoAfterSn-homology}
If
$h_* \colon H_* (S^n \times M; \, \Z) \to H_* (S^n \times M; \, \Z)$
is an isomorphism then
$(p \circ h \circ i)_* \colon H_* (M; \Z) \to H_* (M; \Z)$
is an isomorphism.
\item\label{L_4909_IsoAfterSn-fundamental-group}
If
$h_* \colon \pi_1 (S^n \times M ) \to \pi_1 (S^n \times M)$
is an isomorphism
then $(p \circ h \circ i)_* \colon \pi_1 (M) \to \pi_1 (M)$
is an isomorphism.
(We omit the choice of basepoints in $\pi_1$,
since the spaces in question are path connected.)
\end{enumerate}
\end{lem}

\begin{proof}
We prove~(\ref{L_4909_IsoAfterSn-homology}).
Let $k \in \Nz$;
we show that
\[
(p \circ h \circ i)_* \colon H_k (M; \Z) \to H_k (M; \Z)
\]
is an isomorphism.
For $k > m$,
$H_k (M; \Z) = 0$,
so this is immediate.
Accordingly,
we may assume that $0 \leq k \leq m$.

Let $e_0$ be the usual generator of $H_0 (S^n; \Z) \cong \Z$
and let $e_n$ be a generator of $H_n (S^n; \Z) \cong \Z$.
Since $H_* (S^n; \Z)$ is free,
the K\"{u}nneth formula
(\cite[Theorem 3B.6]{hatcher})
implies that the standard pairing
$(\et, \mu) \mapsto \et \times \mu$
yields a graded isomorphism
\[
\om \colon
 H_* (S^n; \Z) \otimes H_* (M; \Z) \to H_* (S^n \times M; \, \Z).
\]
Since $H_l (M; \Z) = 0$ for $l \geq n$,
it follows that
$\mu \mapsto \om (e_0 \otimes \mu)$
defines an isomorphism
$\bt \colon H_k (M; \Z) \to H_k (S^n \times M; \, \Z)$
(and, similarly,
$\mu \mapsto e_n \times \mu$
defines an isomorphism
$H_k (M; \Z) \to H_{n + k} (S^n \times M; \, \Z)$).
Moreover,
$\bt = i_*$.

Let $p_0 \colon S^n \to \{ y_0 \}$
be the unique map from $S^n$ to $\{ y_0 \}$.
Since $(p_0)_* (e_0)$ is a generator of
$H_0 (\{ y_0 \}; \Z) \cong \Z$,
naturality in the K\"{u}nneth formula
implies that $p_* (e_0 \times \mu) = \mu$
for $\mu \in H_k (M; \Z)$.
(By contrast,
$p_* (e_n \times \mu) = 0$.)
Thus
\[
p_* \colon H_k (S^n \times M; \, \Z) \to H_k (M; \Z)
\]
is an isomorphism.
(In fact,
$p_* = \bt^{-1}$.)

We factor
$(p \circ h \circ i)_* \colon H_k (M; \Z) \to H_k (M; \Z)$
as
\[
H_k (M; \Z)
 \stackrel{i_*}{\longrightarrow} H_k (S^n \times M; \, \Z)
 \stackrel{h_*}{\longrightarrow} H_k (S^n \times M; \, \Z)
 \stackrel{p_*}{\longrightarrow} H_k (M; \, \Z).
\]
We have just shown that the first and last maps are
isomorphisms,
and the middle map is an isomorphism
by hypothesis.
So $(p \circ h \circ i)_*$ is an isomorphism.

Part~(\ref{L_4909_IsoAfterSn-fundamental-group})
follows immediately as soon as we know that
$p_*$ and $i_*$ are isomorphisms.
This fact follows from \cite[Proposition~1.12]{hatcher}.
\end{proof}

We will start our constructions
with the manifold $M$ used in~\cite[Example~3.5]{PhX},
with
$\pi_1 (M)
 \cong \langle a, b \mid a^3 = b^7 = 1, \,\, a b a^{-1} = b^2 \rangle$.
There is a gap in the proof for \cite[Example~3.5]{PhX};
we need to know that there is no automorphism
of $\pi_1 (M)$ which sends the image of~$a$ in the abelianization
to the image of~$a^2$.
We prove that here;
for convenience of the reader and to establish notation in the proof,
we prove all the properties of~$G$ from scratch.

\begin{lem}\label{L_5X17_G}
Let $G$ be the group with presentation
in terms of generators and relations given by
$G = \langle a, b \mid a^3 = b^7 = 1, \,\, a b a^{-1} = b^2 \rangle$.
Then $G$ is a finite group with $21$ elements,
its abelianization is isomorphic to $\Z / 3 \Z$
and is generated by the image of~$a$,
and every automorphism of $G$ induces the identity automorphism
on its abelianization.
\end{lem}

\begin{proof}
Rewrite the last relation as
$a b = b^2 a$.
It follows that for all $r, s \in \Nz$
there is $t \in \Nz$ such that $a^r b^s = b^t a^r$.
Since $a$ and $b$ have finite order,
we therefore have
\begin{equation}\label{Eq_5X17_ListG}
G = \big\{ b^t a^r \colon r, t \in \Nz \big\}.
\end{equation}
Since $a^3 = b^7 = 1$,
it follows that $G$ has at most $21$ elements.

Write $\Z / 3 \Z = \{ 0, 1, 2 \}$
and $\Z / 7 \Z = \{ 0, 1, \ldots, 6 \}$.
One checks that there is an automorphism $\gm$ of $\Z / 7 \Z$
such that $\gm (1) = 2$,
and that $\gm^3 = \id_{\Z / 7 \Z}$.
Thus,
there is a semidirect product group
$S = \Z / 7 \Z \rtimes_{\gm} \Z / 3 \Z$.
Moreover, the elements $(0, 1)$ and $(1, 0)$ satisfy the relations
defining~$G$.
Therefore there is a surjective \hm{} $\ps \colon G \to S$
such that $\ps (a) = (1, 0)$ and $\ps (b) = (0, 1)$.
So $G$ has exactly $21$ elements and
the subgroup $\langle b \rangle \subset G$ is normal
and has order~$7$.

Let $H$ be the abelianization of $G$ and let $\pi \colon G \to H$
be the associated map.
The relations for $G$ show that
there is a surjective \hm{} $\kp \colon G \to \Z / 3 \Z$
such that $\kp (a) = 1$ and $\kp (b) = 0$.
Therefore $\Z / 3 \Z$ is a quotient of~$H$.
Since $\card (G) / \card (H)$ is
prime
and $G$ is not abelian, we get $H \cong \Z / 3 \Z$,
generated by $\pi (a)$.
Moreover, $\pi (b)$ is the identity element of~$H$.

Now let $\ph \colon G \to G$ be an automorphism,
and let ${\ov{\ph}} \colon H \to H$ be the induced automorphism of~$H$.
To show that ${\ov{\ph}} = \id_H$,
we must rule out ${\ov{\ph}} (\pi (a)) = \pi (a^2)$.
So assume ${\ov{\ph}} (\pi (a)) = \pi (a^2)$.
Use (\ref{Eq_5X17_ListG}),
$a^3 = b^7 = 1$, and $\pi (b) = 1$
to find $r \in \{ 0, 1, \ldots, 6 \}$
such that $\ph (a) = b^r a^2$.
Since $\langle b \rangle$ is a normal Sylow $7$-subgroup,
all elements of $G$ of order~$7$
are contained in~$\langle b \rangle$,
so there is $s \in \{ 1, 2, \ldots, 6 \}$ such that $\ph (b) = b^s$.
Apply $\ph$ to the relation $a b a^{-1} = b^2$
to get
\begin{equation}\label{Eq_5X17_Rel}
b^r a^2 b^s a^{-2} b^{-r} = b^{2 s}.
\end{equation}
The relation $a b a^{-1} = b^2$
also implies that $a b^s a^{-1} = b^{2 s}$,
so $a^2 b^s a^{-2} = b^{4 s}$.
Substituting in~(\ref{Eq_5X17_Rel})
gives $b^{4 s} = b^{2 s}$.
Thus $(b^s)^2 = 1$.
Since $\langle b \rangle$ is cyclic of odd order,
we get $b^s = 1$,
so $\ph (b) = 1$,
a contradiction.
\end{proof}

The proof of the following lemma is motivated by ideas
from the proof of Theorem~9 in~\cite{pimsner}.

\begin{lem}\label{L_5Y08_QD}
Let $A$ be a separable continuous trace \ca,
and let $\af \in \Aut (A)$.
Set $X = \Prim (A)$,
let $h \colon X \to X$ be homeomorphism induced by~$\af$
(so that if $P \subset A$ is a primitive ideal,
then $h (P) = \af (P)$),
and assume that $X$ is
an infinite compact metrizable space
and that $h$ is minimal.
Then $A \rtimes_{\af} \Z$ is simple and quasidiagonal.
\end{lem}

\begin{proof}
Simplicity of $A \rtimes_{\af} \Z$
follows from the corollary to Theorem~1 in~\cite{AS}.

We claim that there is a nonzero
\hm{}
\[
\ph \colon
 A \rtimes_{\af} \Z \to C_{\mathrm{b}} (\N, K) / C_0 (\N, K).
\]
Since $A \rtimes_{\af} \Z$ is simple,
it will then follow that $\ph$ is injective.
Since $A$ is nuclear,
so is $A \rtimes_{\af} \Z$.
Therefore we can lift $\ph$ to a completely positive contraction
$T \colon A \rtimes_{\af} \Z \to C_{\mathrm{b}} (\N, \, K)$.
We thus get a sequence $(T_n)_{n \in \N}$
of completely positive contractions
$T_n \colon A \rtimes_{\af} \Z \to K$
such that
\begin{equation}\label{Eq_6119_St}
\lim_{n \to \infty} \| T_n (a b) - T_n (a) T_n (b) \| = 0
\andeqn
\lim_{n \to \infty} \| T_n (a) \| = \| a \|
\end{equation}
for all $a, b \in A \rtimes_{\af} \Z$.
We can pick a sequence $(p_n)_{n \in \N}$ in~$K$
consisting of finite rank projections
such that,
if for all $n \in \N$
we replace $T_n$ by $a \mapsto p_n T_n (x) p_n$,
the resulting sequence of maps still satisfies~(\ref{Eq_6119_St}).
Thus, we may assume that there is a sequence $(l (n))_{n \in \N}$
in~$\N$
such that for all $n \in \N$,
we actually have a completely positive contraction
$T_n \colon A \rtimes_{\af} \Z \to M_{l (n)}$;
moreover, the sequence $(T_n)_{n \in \N}$
satisfies~(\ref{Eq_6119_St}).
Thus $A \rtimes_{\af} \Z$ is quasidiagonal.

It remains to prove the claim.
It suffices to prove the claim
for $A \otimes K$ and $\af \otimes \id_K$
in place of $A$ and~$\af$.
By Proposition 5.59 in~\cite{RW},
we may therefore assume that $A$
is the section algebra of a locally trivial
continuous field $E$ over~$X$
with fiber~$K$.

Fix a point $x_0 \in X$.
Choose a closed neighborhood $S$ of~$x_0$
such that $E |_S$ is trivial,
and let $\kp \colon A \to C (S, K)$
be the composition of the quotient map
$A \to \Gm (E |_S)$ and a trivialization
$\Gm (E |_S) \to C (S, K)$.
For $x \in S$ let $\ev_x \colon C (S, K) \to K$
be evaluation at~$x$.
For $n \in \Z$ define
$\sm_n = \ev_{x_0} \circ \kp \circ \af^n \colon A \to K$,
a surjective \hm{} with kernel $h^{- n} (x_0) \in X = \Prim (A)$.

Since $h$ is minimal,
there is a sequence $(k (n))_{n \in \N}$ in~$\N$
such that
\[
\limi{n} k (n) = \I
\andeqn
\limi{n} h^{- k (n)} (x_0) = x_0.
\]
\Wolog{} $h^{- k (n)} (x_0) \in S$ for all $n \in \N$.
For $n \in \N$,
the \hm{s} $\sm_{k (n)}$
and $\ev_{h^{- k (n)} (x_0)} \circ \kp$
are surjective \hm{s} from $A$ to~$K$
with the same kernel (namely $h^{- k (n)} (x_0) \in X = \Prim (A)$),
so they are unitarily equivalent irreducible representations.
That is,
there is a unitary $v_n \in M (K) = L (l^2)$
such that
\[
v_n \sm_{k (n)} (a) v_n^{-1}
 = \big( \ev_{h^{- k (n)} (x_0)} \circ \kp \big) (a)
\]
for all $a \in A$.
Choose $c_n \in M (K)_{\sa}$
with $\| c_n \| \leq \pi$
such that $v_n = \exp (i c_n)$,
and set $w_n = \exp (i k (n)^{-1} c_n)$.
Then
\[
\| w_n - 1 \| \leq \frac{\pi}{k (n)}
\andeqn
w_n^{k (n)} = v_n.
\]
Define $\rh_n \colon A \to M_{k (n)} (K)$ by
\begin{align*}
& \rh_n (a)
 = \diag \big( \sm_0 (a),
    \, w_n \sm_1 (a) w_n^{-1},
    \, w_n^2 \sm_2 (a) w_n^{-2},
\\
& \hspace*{11em} \mbox{}
    \, \ldots,
    \, w_n^{k (n) - 1} \sm_{k (n) - 1} (a) w_n^{- (k (n) - 1)} \big)
\end{align*}
for $a \in A$.
Further define the permutation unitary $u_n \in M (M_{k (n)} (K))$ by
\[
u_n = \left( \begin{matrix}
  0     &  1     &  0     & \cdots & \cdots &  0        \\
  0     &  0     &  1     & \ddots & \ddots & \vdots    \\
 \vdots & \vdots & \ddots & \ddots & \ddots & \vdots    \\
 \vdots & \vdots & \ddots & \ddots &  1     &  0        \\
  0     &  0     & \cdots & \cdots &  0     &  1        \\
  1     &  0     & \cdots & \cdots &  0     &  0
\end{matrix} \right).
\]

We want to show that for all $a \in A$ we have
$\limi{n} \| \rh_n ( \af (a)) - u_n \rh_n (a) u_n^* \| = 0$.
To do this,
let $a \in A$.
Then
\begin{align*}
& \rh_n ( \af (a))
\\
& \hspace*{1em} \mbox{}
 = \diag \big( \sm_1 (a),
    \, w_n \sm_2 (a) w_n^{-1},
    \, w_n^2 \sm_3 (a) w_n^{-2},
    \, \ldots,
    \, w_n^{k (n) - 1} \sm_{k (n)} (a) w_n^{- (k (n) - 1)} \big)
\\
& \hspace*{1em} \mbox{}
 = u_n \diag \big( w_n^{k (n) - 1} \sm_{k (n)} (a) w_n^{- (k (n) - 1)},
    \, \sm_1 (a),
\\
& \hspace*{6em} \mbox{}
    \, w_n \sm_2 (a) w_n^{-1},
    \, \ldots,
    \, w_n^{k (n) - 2} \sm_{k (n) - 1} (a) w_n^{- (k (n) - 2)}
                   \big) u_n^*.
\end{align*}
Therefore
\begin{align*}
& \| \rh_n ( \af (a)) - u_n \rh_n (a) u_n^* \|
\\
& \mbox{}
 = \| \rh_n (a) - u_n^* \rh_n ( \af (a) ) u_n \|
\\
&
\mbox{}
 = \max \big( \big\| \sm_0 (a)
    - w_n^{k (n) - 1} \sm_{k (n)} (a) w_n^{- (k (n) - 1)} \big\|,
\\
& \hspace*{3em} \mbox{}
     \, \big\| w_n \sm_1 (a) w_n^{-1} - \sm_1 (a) \big\|,
     \, \big\| w_n^2 \sm_2 (a) w_n^{-2} - w_n \sm_2 (a) w_n^{-1} \big\|,
\\
& \hspace*{3em} \mbox{}
     \, \ldots,
     \, \big\| w_n^{k (n) - 1} \sm_{k (n) - 1} (a) w_n^{- (k (n) - 1)}
     - w^{k (n) - 2} \sm_{k (n) - 1} (a) w_n^{- (k (n) - 2)}
                          \big\| \big).
\end{align*}
Every term except the first on the right hand side of this estimate
is dominated by
\[
2 \| w_n - 1 \| \| a \| \leq \frac{2 \pi \| a \|}{k (n)}.
\]
The first term is estimated as follows:
\begin{align*}
& \big\| \sm_0 (a)
    - w_n^{k (n) - 1} \sm_{k (n)} (a) w_n^{- (k (n) - 1)} \big\|
\\
& \hspace*{4em} \mbox{}
 \leq \| \sm_0 (a) - v_n \sm_{k (n)} (a) v_n^* \|
\\
& \hspace*{6em} \mbox{}
   + \big\| w_n^{k (n)} \sm_{k (n)} (a) w_n^{- k (n)}
     - w_n^{k (n) - 1} \sm_{k (n)} (a) w_n^{- (k (n) - 1)} \big\|
\\
& \hspace*{4em} \mbox{}
 \leq \big\| (\ev_{x_0} \circ \kp) (a)
     - \big( \ev_{h^{- k (n)} (x_0)} \circ \kp \big) (a) \big\|
    + 2 \| w_n - 1 \| \| a \|
\\
& \hspace*{4em} \mbox{}
 \leq \big\| (\ev_{x_0} \circ \kp) (a)
     - \big( \ev_{h^{- k (n)} (x_0)} \circ \kp \big) (a) \big\|
    + \frac{2 \pi \| a \|}{k (n)}.
\end{align*}
Now let $\ep > 0$.
Since $\kp (a) \in C (S, K)$ is \ct{}
and $\limi{n} h^{- k (n)} (x_0) = x_0$,
there is $N_1 \in \N$ such that for all $n \geq N_1$
we have
\[
\big\| (\ev_{x_0} \circ \kp) (a)
     - \big( \ev_{h^{- k (n)} (x_0)} \circ \kp \big) (a) \big\|
  < \frac{\ep}{2}.
\]
Since $\limi{n} k (n) = \I$,
there is $N_2 \in \N$ such that for all $n \geq N_2$
we have
\[
\frac{2 \pi \| a \|}{k (n)}
  < \frac{\ep}{2}.
\]
For $n \geq \max (N_1, N_2)$,
we then have $\| \rh_n ( \af (a)) - u_n \rh_n (a) u_n^* \| < \ep$,
as desired.

For $n \in \N$
choose an isomorphism $\ps_n \colon M_{k (n)} (K) \to K$,
and use the same symbol for the induced isomorphism
$M_{k (n)} ( M (K)) \to M (K)$.
Let $u \in M \big( C_{\mathrm{b}} (\N, K) / C_0 (\N, K) \big)$
be the image there of
\[
\big( \ps_1 (u_1), \, \ps_2 (u_2), \, \ldots \big)
  \in C_{\mathrm{b}} (\N, M (K)),
\]
and for $a \in A$ let
$\ps (a) \in C_{\mathrm{b}} (\N, K) / C_0 (\N, K)$
be the image there of
\[
\big( (\ps_1 \circ \rh_1) (a), \, (\ps_2 \circ \rh_2) (a),
    \, \ldots \big)
 \in C_{\mathrm{b}} (\N, K).
\]
Then $u \ps (a) u^* = \ps (\af (a))$ for all $a \in A$,
so $u$ and $\ps$ together define a \hm{}
\[
\ph \colon
 A \rtimes_{\af} \Z \to C_{\mathrm{b}} (\N, K) / C_0 (\N, K).
\]
This \hm{} is nonzero
because if we choose $c \in K \smallsetminus \{ 0 \}$
then there is $a \in A$
such that $\kp (a)$ is the constant function with value~$c$,
and $\| \ps (a) \|$ is easily checked to be~$\| c \|$.
This completes the proof of the claim,
and thus of the lemma.
\end{proof}

To show that the crossed product is quasidiagonal,
it isn't actually necessary that $h$ be minimal.
It suffices to assume that every point of $X$ is chain recurrent.
The basic idea is the same,
but the notation gets messier.

The next lemma shows that for crossed products by automorphisms
with finite Rokhlin dimension, any tracial state on the crossed product
arises from an invariant tracial state on the original algebra.
We refer to~\cite{HP} for a discussion of finite Rokhlin dimension
in the nonunital setting.
(See Definition 1.21 of~\cite{HP}.)

\begin{lem}\label{lem_dimrok-traces}
Let $A$ be a separable \ca,
and let $\alpha \in \Aut (A)$ be an automorphism
with finite Rokhlin dimension.
Let $P \colon A \rtimes_{\alpha} \Z \to A$
be the canonical conditional expectation.
Then for any tracial state $\ta$ on $A \rtimes_{\alpha} \Z$
there is an $\alpha$-invariant tracial state $\rh$ on $A$
such that $\ta = \rh \circ P$.
\end{lem}

\begin{proof}
Let $d \in \Nz$ be the Rokhlin dimension of~$\af$.
Apply the proof of \cite[Proposition 2.8]{HWZ}
and \cite[Remark 2.9]{HWZ}
to Definition 1.21 of~\cite{HP},
to get the following single tower version
of Rokhlin dimension,
in which $d$ is replaced by $2 d + 1$.
For any finite set $F \subset A$,
any $p > 0$, and any $\ep > 0$, there are positive contractions
$f_{0}^{(l)}, \, f_{1}^{(l)}, \, \ldots, f_{p - 1}^{(l)} \in A$
for $l = 0, 1, \ldots, 2 d + 1$, such that:
\begin{enumerate}
\item\label{6119_Orth}
$\big\| f_{j}^{(l)} f_{k}^{(l)} b \big\| < \ep$
for $l = 0, 1, \ldots, 2 d + 1$,
$j, k  =  0, 1, \ldots, p - 1$ with $j \neq k$,
and all $b \in F$.
\item\label{6119_Sum}
$\left\| \left( \sum_{l = 0}^{2 d + 1}
  \sum_{j = 0}^{p - 1} f_{j}^{(l)} \right) b
 - b \right\| < \ep$
for all $b \in F$.
\item\label{6119_Comm}
$\big\| \big[ f_{j}^{(l)}, b \big] \big\| < \ep$
for $l = 0, 1, \ldots, 2 d + 1$,
$j = 0, 1, \ldots, p - 1$,
and all $b \in F$.
\item\label{6119_Perm}
$\big\| \big( \alpha \big( f_{j}^{(l)} \big)
        - f_{j + 1}^{(l)} \big) b \big\| < \ep$ for
$l = 0, 1, \ldots, 2 d + 1$,
$j = 0, 1, \ldots, p - 2$,
and all $b \in F$.
\item\label{6119_Return}
$\big\| \big( \alpha \big(f_{p - 1}^{(l)} \big)
   - f_{0}^{(l)} \big)b \big\| < \ep$
for $l = 0, 1, \ldots, 2 d + 1$
and all $b \in F$.
\setcounter{TmpEnumi}{\value{enumi}}
\end{enumerate}
The argument of Remark 1.18 of~\cite{HP}
shows that we can replace~(\ref{6119_Orth})
by the stronger condition:
\begin{enumerate}
\setcounter{enumi}{\value{TmpEnumi}}
\item\label{6119_ExOrth}
$f_{k}^{(l)} f_{j}^{(l)} = 0$
for $l = 0, 1, \ldots, 2 d + 1$
and
$j, k  =  0, 1, \ldots, p - 1$ with $j \neq k$.
\setcounter{TmpEnumi}{\value{enumi}}
\end{enumerate}

Let $u$ be the canonical unitary in $M (A \rtimes_{\alpha} \Z)$.
Since $A$ contains an approximate identity for $A \rtimes_{\alpha} \Z$,
the restriction $\tau |_{A}$ has norm 1.
Therefore $\tau |_{A}$ is an $\alpha$-invariant tracial state.
Thus, it suffices to show that $\tau (a u^n) = 0$
for all $a \in A$ and for all $n \in \Z \smallsetminus \{ 0 \}$.
We may assume that $\| a \| \leq 1$.
Since
$a u^{- n} = (u^n a^*)^* = [\af^n (a^*) u^n]^*$,
it suffices to treat the case $n > 0$.
Fix $\ep > 0$;
we prove that $| \ta (a u^n) | < \ep$.

Define
\[
\ep_0 = \frac{\ep}{2 n d + n + 1}.
\]
Then $\ep_0 > 0$.
An argument using polynomial approximations
to the function $\ld \mapsto \ld^{1/2}$ on $[0, \I)$
provides $\dt > 0$
such that $\dt \leq \ep_0$ and
whenever $C$ is a \ca{} and $b, c, x \in C$
satisfy
\[
\| b \| \leq 1,
\quad
\| c \| \leq 1,
\quad
\| x \| \leq 1,
\quad
b \geq 0,
\quad
c \geq 0,
\andeqn
\| b x - x c \| < \dt,
\]
then $\big\| b^{1/2} x - x c^{1/2} \big\| < \ep_0$.
Set $\dt_0 = \dt / (n + 1)$.

Apply the single tower property above
with (\ref{6119_ExOrth}) in place of~(\ref{6119_Orth}),
with $p = n + 1$, with
\[
F = \big\{ a, \, a^*, \, \af^{- 1} (a^*), \, \af^{- 2} (a^*),
   \, \ldots, \, \af^{- (n - 1)} (a^*) \big\},
\]
and with $\dt_0$ in place of~$\ep$,
getting positive contractions
$f_{0}^{(l)}, \, f_{1}^{(l)}, \, \ldots, f_{n}^{(l)} \in A$
as above.
In particular, whenever $j \neq k$
we have $f_{j}^{(l)} f_{k}^{(l)} = 0$,
so
$\big( f_{j}^{(l)} \big)^{1/2} \big( f_{k}^{(l)} \big)^{1/2} = 0$.

In the following estimates,
we interpret all subscripts in expressions $f_{k}^{(l)}$
as elements of $\{ 0, 1, \ldots, n \}$
by reduction modulo $n + 1$.
For $l = 0, 1, \ldots, 2 d + 1$,
for $k = 0, 1, \ldots, p - 1$,
and for $b \in A$,
we have
\[
\big\| \big( \af^n \big( f_{k - n}^{(l)} \big)
        - f_{k}^{(l)} \big) b \big\|
 \leq \sum_{m = 1}^n
   \big\| \af^{n - m} \big( \big(
                 \af \big( f_{k - n + m - 1}^{(l)} \big)
        - f_{k - n + m}^{(l)} \big) \af^{m - n} (b) \big) \big\|.
\]
Putting $b = a^*$,
and using (\ref{6119_Perm}), (\ref{6119_Return}),
and the definition of~$F$,
it follows that
\[
\big\| \left( \af^n \big( f_{k - n}^{(l)} \big)
        - f_{k}^{(l)} \right) a^* \big\| < n \dt_0.
\]
Using
$u^n b u^{- n} = \af^{n} (b)$ for $b \in A$
and taking adjoints,
we get
\[
\big\| a \big( f_{k}^{(l)}
        - u^{n} f_{k - n}^{(l)} u^{- n} \big) \big\| < n \dt_0.
\]
Therefore,
also using~(\ref{6119_Comm}),
\begin{align*}
& \big\| f_{k}^{(l)} a u^n - a u^n f_{k - n}^{(l)} \big\|
\\
& \hspace*{3em} {\mbox{}}
\leq \big\| f_{k}^{(l)} a - a f_{k}^{(l)} \big\| \| u^n \|
  + \big\| a \big( f_{k}^{(l)}
         - u^n f_{k - n}^{(l)} u^{- n} \big) \big\| \| u^n \|
\\
& \hspace*{3em} {\mbox{}}
< \dt_0 + n \dt_0
= \dt.
\end{align*}
Using the choice of $\dt$ and $\big\| f_k^{(l)} \big\| \leq 1$
at the second step,
we now get
\begin{align*}
& \big\| f_k^{(l)} a u^n
- \big( f_k^{(l)} \big)^{1/2} a u^n \big( f_k^{(l)} \big)^{1/2} \big\|
\\
& \hspace*{3em} {\mbox{}}
\leq \big\| \big( f_k^{(l)} \big)^{1/2} \big\|
  \cdot \big\| \big( f_k^{(l)} \big)^{1/2} a u^n
            - a u^n \big( f_k^{(l)} \big)^{1/2} \big\|
   < \ep_0.
\end{align*}
Using the trace property at the second step
and, at the third step,
the fact that $k$ and $k - n$ are not equal modulo $n + 1$,
we now get
\begin{align*}
\big| \tau \big( f_k^{(l)} a u^n \big) \big|
& < \big| \tau \big( \big( f_k^{(l)} \big)^{1/2}
        a u^n \big( f_{k - n}^{(l)} \big)^{1/2} \big) \big| + \ep_0
\\
& = \big| \tau \big( \big(f_{k - n}^{(l)} \big)^{1/2}
        \big(f_k^{(l)} \big)^{1/2} a u^n \big) \big|
      + \ep_0
  = \ep_0.
\end{align*}
Using~(\ref{6119_Sum}) and $\dt_0 \leq \ep_0$,
we therefore get
\[
| \tau (a u^n) |
 < \left| \tau \left ( \sum_{l = 0}^{2 d + 1}
       \sum_{j = 0}^{n} f_{j}^{(l)} a u^n \right ) \right|
    + \dt_0
 < (2 d + 1) n \ep_0 + \ep_0
 = \ep.
\]
This completes the proof.
\end{proof}

\section{Constructing the examples}\label{Sec_Ex}

In this section,
we construct our examples.

\begin{thm}\label{T_5X17_AHEx}
There exist a simple unital separable AH-algebra~$A$
with a unique tracial state
and satisfying $A \cong A \otimes M_{2^{\infty}}$,
and a continuous action $\gamma \colon \T \to \Aut (A)$,
with the following properties:
\begin{enumerate}
\item\label{T_5X17_AHEx_FP}
The fixed point subalgebra $A^{\gamma}$
is not isomorphic to its opposite.
\item\label{T_5X17_AHEx_CP}
The crossed product $A \rtimes_{\gamma} \T$
is not isomorphic to its opposite.
\item\label{T_5X17_AHEx_Eq}
The \ca{} $A$ is not $\T$-equivariantly isomorphic
to its opposite.
\end{enumerate}
\end{thm}

\begin{thm}\label{T_5X17_KbgEx}
There exist a unital Kirchberg algebra $B$
satisfying the Universal Coefficient Theorem,
and a continuous action $\gamma \colon \T \to \Aut (B)$,
with the following properties:
\begin{enumerate}
\item\label{T_5X17_KbgEx_FP}
The fixed point subalgebra $B^{\gamma}$
is not isomorphic to its opposite.
\item\label{T_5X17_KbgEx_CP}
The crossed product $B \rtimes_{\gamma} \T$
is not isomorphic to its opposite.
\item\label{T_5X17_KbgEx_Eq}
The \ca{} $B$ is not $\T$-equivariantly isomorphic
to its opposite.
\end{enumerate}
\end{thm}

The proofs of Theorem~\ref{T_5X17_AHEx}
and Theorem~\ref{T_5X17_KbgEx}
are the same until nearly the end,
so we prove them together.

\begin{proof}[Proofs of Theorem~\ref{T_5X17_AHEx}
  and Theorem~\ref{T_5X17_KbgEx}]
We start with the compact
connected manifold $M$ used in~\cite[Example~3.5]{PhX},
whose fundamental group can be identified with the group~$G$
of Lemma~\ref{L_5X17_G}
and whose signature is nonzero.
It follows from
\cite[Theorem 2A.1]{hatcher}
that $H_1 (M; \Z) \cong \Z / 3 \Z$,
generated by the image of~$a$
under the map
$\pi_1 (M) \to H_1 (M; \Z)$,
so Poincar\'{e} duality
(\cite[Theorem 3.30]{hatcher})
gives
$H^3 (M; \Z) \cong \Z / 3 \Z$.

Let $\eta \in H^3 (M; \Z)$ be a generator.
We claim that if $h \colon M \to M$ is a continuous map
such that the induced maps
$h_* \colon H_* (M; \Z) \to H_* (M; \Z)$
and $h_* \colon \pi_1 (M) \to \pi_1 (M)$ are isomorphisms,
then $h^* (\eta) = \eta$.

To prove the claim,
note first that the Universal Coefficient Theorem
and the Five Lemma imply that
$h^* \colon  H^* (M; \Z) \to H^* (M; \Z)$ is an isomorphism as well.
By Lemma \ref{L_4909_SameSig},
$h$ is orientation preserving.
Since $h_* \colon \pi_1 (M) \to \pi_1 (M)$ is an automorphism,
and there is no automorphism of $\pi_1 (M)$ which sends $a$ to $a^{-1}$,
it follows by naturality of the Hurweicz map
that $h_* \colon H_1 (M; \Z) \to H_1 (M; \Z)$ is the identity.
Since $h_*$ fixes the orientation class,
it follows by naturality in Poincar\'{e} duality
that $h^*$ is also the identity on $H^3 (M; \Z)$,
as required.

To proceed,
we would have liked to have a minimal homeomorphism of~$M$.
In Remark~\ref{rmk:why-we-need-large-sphere} below,
we explain why no such homeomorphism exists.
We remedy this situation by giving ourselves more space,
as follows.

Choose an odd integer $n \geq 5$.
Let $\eta_0 \in H^3 (S^n \times M; \, \Z)$
be the product of the standard generator of $H^0 (S^n; \Z)$
and~$\eta$.
Let $h \colon S^n \times M \to S^n \times M$
be a \cfn{}
such that the induced maps
\[
h_* \colon H_* (S^n \times M; \, \Z) \to H_* (S^n \times M; \, \Z)
\]
and
\[
h_* \colon \pi_1 (S^n \times M) \to \pi_1 (S^n \times M)
\]
are isomorphisms.
Applying
Lemma~\ref{L_4909_IsoAfterSn}
and following the notation there,
the induced maps
\[
(p \circ h \circ i)_* \colon H_* (M; \Z) \to H_* (M; \Z)
\andeqn
(p \circ h \circ i)_* \colon \pi_1 (M) \to \pi_1 (M)
\]
are isomorphisms.
By the claim above, $(p \circ h \circ i)^* (\eta) = \eta$.
It follows from the definitions of the maps and of~$\eta_0$
that $h^* (\eta_0) = \eta_0$.

By \cite[Corollary 1.7]{grothendieck},
there exist
$N \in \N$
and a locally trivial \ct{} field $E$ over $S^n \times M$
with fiber $M_N$
whose section algebra $\Gm (E)$ has Dixmier-Douady invariant~$\eta_0$.
We identify $\Prim (\Gm (E))$ with $S^n \times M$ in the obvious way.
Since $\Gm (E)^{\op}$ has Dixmier-Douady invariant $-\eta_0$,
an isomorphism from $\Gm (E) \otimes K$
to $\Gm (E)^{\op} \otimes K$
would induce a homeomorphism from $\Prim (\Gm (E))$
to itself
whose induced action on $H^3 (S^n \times M)$
sends $\eta_0$ to $- \eta_0$.
Since no such homeomorphism exists,
it follows that $\Gm (E)$ is not stably isomorphic
to its opposite algebra.

Since $S^n$ admits a free action of~$\T$,
so does $S^n \times M$.
By \cite[Theorem 1 and Theorem 4]{FH},
there exists a uniquely ergodic minimal diffeomorphism
$h \colon S^n \times M \to S^n \times M$
which is homotopic to the identity.
Thus
$h^* (E) \cong E$.
Therefore there exists an automorphism
$\alpha \colon \Gm (E) \to \Gm (E)$
which induces $h$ on $\Prim (\Gm (E))$.
Set
$A_0 = \Gm (E) \rtimes_{\alpha} \Z$.
Then
$A_0$ is a separable unital
nuclear \ca{} satisfying the Universal Coefficient Theorem.
The algebra $A_0$ is simple and quasidiagonal by Lemma~\ref{L_5Y08_QD}.

We claim that $A_0$ has finite nuclear dimension
and a unique tracial state.
To prove the claim,
use \cite[Corollary 3.10]{kirchberg-winter}
to see that the decomposition rank of $\Gm (E)$
is $\dim (S^n \times M)$, hence finite.
Therefore
$\Gm (E)$ has finite nuclear dimension.
Since the center of $\Gm (E)$ is isomorphic to $C (S^n \times M)$,
it follows that $\alpha|_{Z (\Gm (E))}$
is an automorphism of $C (S^n \times M)$
arising from a minimal homeomorphism.
Therefore,
by \cite[Theorem 6.1]{HWZ} or \cite[Corollary 2.6]{szabo},
$\alpha|_{Z (\Gm (E))}$ has finite Rokhlin dimension.
It follows immediately from the definition of finite Rokhlin dimension
that if $\alpha$ is an automorphism of a \ca~$C$
and the restriction of $\alpha$ to the center of $C$
has finite Rokhlin dimension,
then $\alpha$ has finite Rokhlin dimension as well
(with commuting towers).
Therefore
$A_0$ has finite nuclear dimension
by \cite[Theorem 4.1]{HWZ}.
Since $h$ uniquely ergodic,
$\Gamma (E)$ admits a unique invariant tracial state,
and thus, by Lemma \ref{lem_dimrok-traces},
$A_0$ has a unique tracial state.

The fixed point subalgebra of the dual action $\gamma$ of $\T$ on~$A_0$
is isomorphic to~$\Gm (E)$.
Therefore
it is not isomorphic
to its opposite algebra.
By Takai duality,
the crossed product of $A_0$ by the dual action
is stably isomorphic to $\Gamma (E)$,
and therefore also not isomorphic to its opposite.
In general,
if two \ca{s} with a $G$-action are equivariantly isomorphic,
it follows immediately that
their fixed point subalgebras are isomorphic.
Thus,
$(\T, A_0, \gamma)$ is not equivariantly isomorphic
to its opposite $(\T, A_0^{\op}, \gamma^{\op})$.

The remainder of the proof consists of showing that
the same properties remain
after we tensor everything with $\mathcal{O}_{\infty}$
(for Theorem~\ref{T_5X17_KbgEx})
or with $M_{2^{\infty}}$
(for Theorem~\ref{T_5X17_AHEx}).

For any continuous field $F$ over~$X$
and any nuclear \ca~$D$,
denote by $F \otimes D$
the continuous field whose fiber over $x \in X$
is $F_x \otimes D$.
(This is in fact a continuous field by \cite[Theorem 4.5]{KW1}.)
Suppose that $F_1$ and $F_2$ are two continuous fields over $X$
with fibers $M_N$,
and that
$\Gm ( F_1 \otimes {\mathcal{O}}_{\infty} \otimes K)
 \cong \Gm (F_2 \otimes {\mathcal{O}}_{\infty} \otimes K)$.
Since the fibers of these fields are simple,
it follows that there is a homeomorphism $g \colon X \to X$
such that
$g^* (F_2 \otimes {\mathcal{O}}_{\infty} \otimes K)
  \cong F_1 \otimes {\mathcal{O}}_{\infty} \otimes K$.
Apply \cite[Corollary~4.9]{paper:DadarlatPennigSpectra},
noting that $\C$ is included among
the strongly selfabsorbing \ca{s} there.
(See the beginning of \cite[Section~2.1]{paper:DadarlatPennigSpectra}.)
We conclude that $g^* (F_2) \otimes K \cong F_1 \otimes K$,
so $\Gm (F_2) \otimes K \cong \Gm (F_1) \otimes K$.
Taking $F_1 = E$ and $F_2 = E^{\#}$
(the fiberwise conjugate field,
with fibers $(E^{\#})_x = (E_x)^{\#}$),
the fact that $\Gm (E)$ is not stably isomorphic to its opposite
algebra now gives the second step of the following calculation,
while
$({\mathcal{O}}_{\infty} \otimes K)^{\#}
 \cong {\mathcal{O}}_{\infty} \otimes K$
gives the third step:
\begin{align}\label{Eq_5Y11_IsoOpp}
\Gm (E) \otimes {\mathcal{O}}_{\infty} \otimes K
& \cong \Gm (E \otimes {\mathcal{O}}_{\infty} \otimes K)
\\
& \not\cong \Gm (E^{\#} \otimes {\mathcal{O}}_{\infty} \otimes K)
\notag
\\
& \cong
   \Gm \big( E^{\#} \otimes ({\mathcal{O}}_{\infty}
                                      \otimes K)^{\#} \big)
  \cong \big( \Gm (E) \otimes {\mathcal{O}}_{\infty}
                                      \otimes K \big)^{\#}.
\notag
\end{align}

Set $B = A_0 \otimes {\mathcal{O}}_{\infty}$
and let $\gm \colon \T \to \Aut (B)$
be the tensor product of the dual action on~$A_0$
and the trivial action on~${\mathcal{O}}_{\infty}$.
Then
\[
B \rtimes_{\gm} \T
   \cong \Gm (E) \otimes {\mathcal{O}}_{\infty} \otimes K,
\quad
B^{\#} \rtimes_{\gm^{\#}} \T
 \cong \big( \Gm (E) \otimes
             {\mathcal{O}}_{\infty} \otimes K \big)^{\#},
\]
\[
B^{\gm} \cong \Gm (E) \otimes {\mathcal{O}}_{\infty},
\quad
{\mbox{and}}
\quad
(B^{\#})^{\gm^{\#}} \cong (\Gm (E) \otimes {\mathcal{O}}_{\infty})^{\#}.
\]
So parts (\ref{T_5X17_KbgEx_FP}) and~(\ref{T_5X17_KbgEx_CP})
of Theorem~\ref{T_5X17_KbgEx} follow from~(\ref{Eq_5Y11_IsoOpp})
and Lemma~\ref{L_5Y11_Conj}.
Part~(\ref{T_5X17_KbgEx_Eq}) is now immediate.
Since $B$ is a unital Kirchberg algebra which satisfies the
Universal Coefficient Theorem,
we have proved Theorem~\ref{T_5X17_KbgEx}.

Next we are going to show that
if we tensor $E$ fiberwise with the CAR algebra $M_{2^{\infty}}$,
the section algebra
will still fail to be stably isomorphic to its opposite algebra.

Set $D = M_{2^{\infty}}$
and let ${\overline{E}}^*_D (X)$ be the (reduced) cohomology theory
which arises
as in \cite[Corollary 3.9]{paper:DadarlatPennigSpectra}
from the infinite loop structure
of the classifying space of $\Aut_0 (D \otimes K)$,
the component of the identity
of the automorphism group
of $D \otimes K$.
As in \cite[Corollary 3.9]{paper:DadarlatPennigSpectra},
locally trivial bundles with fiber $D \otimes K$
and structure group $\Aut_0 (D \otimes K)$
over a finite connected CW~complex~$X$
are classified by the group ${\overline{E}}^1_D (X)$.

As in the proof of Corollary~4.8
in~\cite{paper:DadarlatPennigSpectra},
the unital map $\mathbb{C} \to D$
induces a morphism $\Aut_0 (K) \to \Aut_0 (D \otimes K)$ and a natural
transformation of cohomology theories
$T \colon
 {\overline{E}}^*_{\mathbb{C}} (X) \to {\overline{E}}^*_{D} (X)$.
Let $t \colon H^3 (X; {\mathbb{Z}}) \to H^3 (X; \, {\mathbb{Z}} [1/2])$
be the coefficient map induced by
\[
{\mathbb{Z}}
 \stackrel{\cong}{\longrightarrow} \pi_2 (\mathrm{Aut}_0 (K))
 \longrightarrow \pi_2 (\mathrm{Aut}_0 (D \otimes K))
 \stackrel{\cong}{\longrightarrow} K_0 (D)
 \stackrel{\cong}{\longrightarrow} {\mathbb{Z}} [1/2].
\]
Using the naturality of the Atiyah-Hirzebruch spectral sequence,
it was furthermore shown in the proof of Corollary~4.8
in~\cite{paper:DadarlatPennigSpectra}
that there is a commutative diagram
\[
\xymatrix{\overline{E}^1_{\mathbb{C}} (X)
  \ar[r]^{T} \ar[d]_{\overline{\delta}_1}
 & \overline{E}^1_{D} (X)
  \ar[d]^{\overline{\delta}_1}
\\
H^3 (X; {\mathbb{Z}}) \ar[r]^-{t}
 & H^3 (X; \, {\mathbb{Z}} [1/2])}
\]
in which the vertical maps are the edge homomorphisms.
The first vertical map is an isomorphism and it can be
identified with the Dixmier-Douady map.
In particular, $T$~is injective whenever $t$ is injective.

In the case of $X = S^n \times M$ with $M$ and $n$ as above,
$H^3 (X; {\mathbb{Z}}) \cong {\mathbb{Z}} /3 {\mathbb{Z}}$.
Since $X$ is a compact manifold,
its integral cohomology is finitely generated.
It therefore follows
from the cohomology Universal Coefficient Theorem
given for chain complexes
in Theorem~10 in Section~5 of Chapter~5 of~\cite{Sp}
that $t$ is bijective.
Hence the map
$T \colon
 {\overline{E}}^1_{\mathbb{C}} (X) \to {\overline{E}}^1_{D} (X)$
is injective.

Now suppose that $F_1$ and $F_2$ are two continuous fields over $X$
with fibers $M_N$,
and that
$\Gm ( F_1 \otimes D \otimes K) \cong \Gm (F_2 \otimes D \otimes K)$.
As in the argument above for the case $D = {\mathcal{O}}_{\infty}$,
there is a homeomorphism $h \colon X \to X$
such that
$h^* (F_2 \otimes D \otimes K) \cong F_1 \otimes D \otimes K$.
Since ${\overline{E}}^1_{\mathbb{C}} (X) \to {\overline{E}}^1_{D} (X)$
is injective,
it follows that $h^* (F_2 \otimes K) \cong F_1 \otimes K$.
Taking $F_1 = E$ and $F_2 = E^{\#}$ as before,
we deduce as before that
\[
\Gm (E) \otimes D \otimes K
 \not\cong  \big( \Gm (E) \otimes D \otimes K \big)^{\#}.
\]
Now define $A = A_0 \otimes D$
and let $\gm \colon \T \to \Aut (A)$
be the tensor product of the dual action on~$A_0$
and the trivial action on~$D$.
Proceed as before to deduce
parts (\ref{T_5X17_AHEx_FP}), (\ref{T_5X17_AHEx_CP}),
and~(\ref{T_5X17_AHEx_Eq})
of Theorem~\ref{T_5X17_AHEx}.

Since $A = A_0 \otimes M_{2^{\infty}}$,
it follows from \cite[Theorem 6.1]{Matui-Sato}
that $A$ is tracially AF
and from \cite[Corollary 6.1]{Matui-Sato}
that $A$ isomorphic to an AH-algebra
with real rank zero and no dimension growth.
This concludes the proof of Theorem~\ref{T_5X17_AHEx}.
\end{proof}

\begin{cor}\label{C_5Y07_AHEx}
There exist a simple unital separable AH-algebra~$A$
with a unique tracial state
and satisfying $A \cong A \otimes M_{2^{\infty}}$,
and an  automorphism $\alpha \in \Aut (A)$
such that $\alpha^n$ is outer for all $n \neq 0$,
with the following properties:
\begin{enumerate}
\item\label{C_5Y07_AHEx_FP}
The fixed point subalgebra $A^{\alpha}$
is not isomorphic to its opposite.
\item\label{C_5Y07_AHEx_Eq}
The \ca~$A$ is not $\Z$-equivariantly isomorphic
to its opposite.
\end{enumerate}
\end{cor}

\begin{cor}\label{C_5Y07_KbgEx}
There exist a unital Kirchberg algebra $B$
satisfying the Universal Coefficient Theorem
and an automorphism $\alpha \in \Aut (B)$
such that $\alpha^n$ is outer for all $n\neq 0$,
with the following properties:
\begin{enumerate}
\item\label{C_5Y07_KbgEx_FP}
The fixed point subalgebra $B^{\alpha}$
is not isomorphic to its opposite.
\item\label{C_5Y07_KbgEx_Eq}
The \ca~$B$ is not $\Z$-equivariantly isomorphic
to its opposite.
\end{enumerate}
\end{cor}

\begin{proof}[Proofs of Corollary~\ref{C_5Y07_AHEx}
  and Corollary~\ref{C_5Y07_KbgEx}]
The proofs of both corollaries are the same.
Let $\gamma \colon \T \to \Aut (A)$
or $\gamma \colon \T \to \Aut (B)$
be the circle action from Theorem~\ref{T_5X17_AHEx}
or Theorem~\ref{T_5X17_KbgEx} as appropriate.
Let $\zt \in \T$ be an irrational angle,
so that $\Z \cdot \zt$ is dense in $\T$.
Set $\alpha = \gamma_{\zt}$.
Then $A^{\gamma} = A^{\alpha}$
or $B^{\gamma} = B^{\alpha}$.
If $\alpha$ is chosen suitably, then $\alpha^n$ will be outer for all $n\neq 0$.
Such choices exist by Lemma~\ref{L_5Y08_Outer}.
\end{proof}

In these corollaries,
we do not claim that the crossed products are not isomorphic.
In particular,
for the actions used
in the proof of Corollary~\ref{C_5Y07_KbgEx},
we will show that the crossed products
actually are at least sometimes isomorphic;
probably this is true in general.
We need a lemma, which state in greater generality than we need.


\begin{lem}\label{L_5Y08_Outer}
Let $A$ be a separable unital \ca.
Let $\alpha \in \Aut (A)$.
Suppose $A$ has a faithful invariant tracial state~$\tau$.
Let $\gm \colon \T \to \Aut ( A \rtimes_{\af} \Z )$
be the dual action on the crossed product.
Then for all but countably many $\ld \in \T$,
the automorphism $\gm_{\ld}$ is outer.
\end{lem}

\begin{proof}
Let $\pi \colon A \to L (H)$ be the Gelfand-Naimark-Segal
representation associated with~$\ta$,
and let $\xi \in H$ be the associated cyclic vector.
Using the $\af$-invariance of~$\ta$,
we find that
\[
\big\langle \pi (\af (a)) \xi, \, \pi (\af (b)) \xi \big\rangle
 = \big\langle \pi (a) \xi, \, \pi (b) \xi \big\rangle
\]
for all $a, b \in A$,
from which it follows that there is a unique isometry
$s \in L (H)$
such that $s \pi (a) \xi = \pi (\af (a)) \xi$
for all $a \in A$.
Applying the same argument with $\af^{-1}$ in place of~$\af$,
we find that $s$ is unitary.

Let $\ld \in \T$.
We claim that if $\gm_{\ld}$ is inner,
then ${\overline{\ld}}$ is an eigenvalue of~$s$.
Since $H$ is separable,
$s$ has at most countably many eigenvalues,
and the lemma will follow.

To prove the claim,
suppose there is $v \in A \rtimes_{\af} \Z$
such that $\gm_{\ld} (b) = v b v^*$
for all $b \in A \rtimes_{\af} \Z$.
Let $Q \colon A \rtimes_{\af} \Z \to A$
be the standard conditional expectation.
Let $u \in A \rtimes_{\af} \Z$
be the canonical unitary of the crossed product,
so that $u a u^* = \af (a)$ for all $a \in A$.
Then
\begin{equation}\label{Eq_5Y14_Cj}
u^* v u
 = u^* (v u v^*) v
 = u^* (\ld u) v
 = \ld v.
\end{equation}
For $n \in \Z$ let $a_n = Q (v u^{-n}) \in A$
be the $n$-th coefficient of~$v$
in the crossed product,
so that
(see \cite[Theorem 8.2.2]{Dv}) $v$ is given by the
limit of the Ces\`{a}ro means:
\[
v = \lim_{n \to \infty}
   \sum_{k = - n}^n \left( 1 - \frac{| k |}{n} \right) a_n u^n.
\]
Applying~(\ref{Eq_5Y14_Cj}) and $u^* a_n u = \af^{-1} (a_n)$,
we get
\[
\ld v
 = u^* v u
 = \lim_{n \to \infty}
 \sum_{k = - n}^n \left( 1 - \frac{| k |}{n} \right) \af^{-1} (a_n) u^n.
\]
It follows that
for all $n \in \Z$ we have
\[
\ld a_n
 = Q (\ld v u^{-n})
 = \af^{-1} (a_n).
\]
Choose $n \in \Z$ such that $a_n \neq 0$.
We have
$\langle \pi (a_n) \xi, \, \pi (a_n) \xi \rangle = \ta (a_n^* a_n)$,
which is nonzero because $\ta$ is faithful.
Therefore $\pi (a_n) \xi \in H$ is nonzero and satisfies
$s^* \pi (a_n) \xi = \pi (\af^{-1} (a_n) ) \xi = \ld \pi (a_n) \xi$.
\end{proof}

Lemma~\ref{L_5Y08_Outer} applies to our setting as follows.
Let $X$ be a \cms, let $n \in \N$,
let $E$ be a locally trivial bundle over~$X$ with fiber~$M_n$,
let $h \colon X \to X$ be a \mh,
and let $\af \in \Aut (\Gm (E))$
be an automorphism which induces the map $h$ on~$\Prim (A)$.
Set $A = \Gm (E)$ and $B = A \rtimes_{\af} \Z$.
Let $\mu$ be an $h$-invariant Borel probability measure on~$X$.
For $x \in X$, since $E_x \cong M_n$,
we can let $\tr_x \colon E_x \to \C$ be the normalized trace.
Then there is a conditional expectation $P \colon A \to C (X)$
such that $P (a) (x) = \tr (a (x))$ for all $x \in X$.
Define $\ta \colon A \to \C$
by $\ta (a) = \int_X P (a) \, d \mu$
for $a \in A$.
Then $\ta$ is an $\af$-invariant tracial state.
Since $h$ is minimal,
$\mu$ has full support.
Therefore $\ta$
is faithful.
It follows from Lemma~\ref{L_5Y08_Outer}
that for all but countably many choices of $\zt \in \T$
in the proof of Corollary~\ref{C_5Y07_KbgEx},
the automorphism $\gm_{\zt}$ used there is outer.

We use \cite[Theorem~1]{Ws}
to see that the tensor product of any automorphism with an
outer automorphism is outer.
If $B$ is a unital Kirchberg algebra
satisfying the Universal Coefficient Theorem,
and $\af^n$ is outer for all $n \in \Z \smallsetminus \{ 0 \}$,
then the crossed product $B \rtimes_{\alpha} \Z$
is also a Kirchberg algebra.
(Pure infiniteness follows from \cite[Corollary 4.4]{JO}.)
By the Five Lemma,
$B \rtimes_{\alpha} \Z$ and $B^{\op} \rtimes_{\alpha} \Z$
have the same $K$-theory,
so they are isomorphic.

It seems very likely that suitable generalizations
of Theorem~12 in Section~V of~\cite{Ex2}
and Theorem~11 in Section~VI of~\cite{Ex2}
will show that,
in the proof of Corollary~\ref{C_5Y07_KbgEx},
the automorphism $\gm_{\zt}$ is outer
for all $\zt \not\in \exp (2 \pi i \Q)$.
The results of~\cite{Ex2} are stated for automorphisms of $C (X)$
for connected compact spaces~$X$,
and one would need to generalize them to
automorphisms of section algebras of locally trivial $M_n$-bundles
over such spaces.

\section{Remarks and questions}\label{Sec_Diff}

We collect here several remarks:
we show that the manifold $M$ used in
the proofs above does not itself
admit any minimal homeomorphisms,
and we describe a shorter construction
of examples,
with the disadvantages that it does not give unital algebras
and that we don't have proofs of some of the extra properties
of the examples.
We finish with several open questions.

\begin{rmk}\label{rmk:why-we-need-large-sphere}
We explain here why the manifold $M$ we started with
in the proofs of Theorem~\ref{T_5X17_AHEx}
and Theorem~\ref{T_5X17_KbgEx}
does not admit minimal homeomorphisms.

We first make the following purely algebraic claim:
if $a \in M_n (\R)$,
then there exists $k \in \{1, 2, \ldots, n + 1\}$
such that $\Tr (a^k) \geq 0$.
We are indebted to Ilya Tyomkin for providing us with the argument.
Assume for contradiction that $\Tr (a) < 0$
for $k = 1, 2, \ldots, n + 1$.
Define polynomials $e_m (t_1, t_2, \ldots, t_n)$
and $p_m (t_1, t_2, \ldots, t_n)$
of $n$ variables $t_1, t_2, \ldots, t_n$
as follows.
For $m = 0, 1, \ldots, n$,
take $e_m$ be the $m$-th elementary symmetric function
(\cite[page~19]{Mcd};
the formulas in~\cite{Mcd} are actually written
in terms of
formal infinite linear combinations of monomials in infinitely
many variables,
and we use the result of setting $t_{n + 1} = t_{n + 2} = \cdots = 0$).
For $m = 1, 2, \ldots, n$, set
$p_m (t_1, t_2, \ldots, t_n) = \sum_{k = 1}^n t_k^m$
(\cite[page~23]{Mcd}).
Newton's formula
(equation (2.11$'$) on page~23 of~\cite{Mcd})
states that
\begin{equation}\label{Eq_5Y07_Newton}
m e_m (t_1, t_2, \ldots, t_n)
 = \sum_{r = 1}^m (-1)^{r - 1} p_r (t_1, t_2, \ldots, t_n)
                e_{m - r} (t_1, t_2, \ldots, t_n)
\end{equation}
for $m = 1, 2, \ldots, n$.

Now let $\lambda_1, \lambda_2, \ldots, \lambda_n$
be the eigenvalues of~$a$,
counting multiplicity.
Then $\Tr (a^k) = p_k (\lambda_1, \lambda_2, \ldots, \lambda_n)$
and the characteristic polynomial of~$a$ is
\[
q (x) = \sum_{k = 0}^n (-1)^k
  e_k (\lambda_1, \lambda_2, \ldots, \lambda_n) x^{n - k}.
\]
Our assumption implies that
$p_k (\lambda_1, \lambda_2, \ldots, \lambda_n) < 0$
for $m = 1, 2, \ldots, n$.
An induction argument using~(\ref{Eq_5Y07_Newton})
shows that $(-1)^k e_k (\lambda_1, \lambda_2, \ldots, \lambda_n) > 0$
for $k = 0, 1, \ldots, n$.
Therefore
\[
\Tr (a q (a))
 = \sum_{k = 0}^n (-1)^k
  e_k (\lambda_1, \lambda_2, \ldots, \lambda_n) \Tr (a^{n - k + 1})
 < 0.
\]
But the Cayley-Hamilton theorem implies that
$q (a) = 0$,
so
$\Tr (a q (a)) = 0$,
a contradiction.
This proves the claim.

Now let $h \colon M \to M$ be a homeomorphism.
We claim that $h$ has a periodic point,
and therefore cannot be minimal.
The groups $H_1 (M; \Q)$ and $H_3 (M; \Q)$ are trivial.
Since $M$ has nonzero signature,
Lemma~\ref{L_4909_SameSig} implies that
$h$ is orientation preserving.
So $h$ acts as the identity on $H_0 (M; \Q)$ and $H_4 (M; \Q)$.
Now
$H_2 (M; \Q)$ is a finite dimensional vector space over~$\Q$.
Therefore,
by the claim above,
there exists some $k > 0$ such that
the map $h_* \colon H_2 (M; \Q) \to H_2 (M; \Q)$
satisfies $\Tr \big( (h_*)^k \big) \geq 0$.
It now follows from the Lefschetz fixed point theorem
that $h^k$ has a fixed point,
as claimed.
\end{rmk}

\begin{rmk}\label{R_5Y07_Alt}
We describe a different method to construct an example
as in Theorem~\ref{T_5X17_AHEx}.
The argument is shorter and does not rely
on the existence theorem of~\cite{FH} to produce
a minimal homeomorphism,
but has the disadvantage that the
resulting algebra is not unital.
In particular,
we do not get the detailed properties
given in Theorem~\ref{T_5X17_AHEx},
because the results needed to get them
are not known in the nonunital case.

Fix $n \in \N$ with $n \geq 15$.
Set $X = {\mathbb{T}}^{n}$.
Choose a uniquely ergodic minimal homeomorphism $h \colon X \to X$
which is homotopic to $\id_X$.
(For example,
choose $\te_1, \te_2, \ldots, \te_n \in \R$
such that $1, \te_1, \te_2, \ldots, \te_n$ are linearly independent
over~$\Q$,
and define
\[
h ( \zt_1, \zt_2, \ldots, \zt_n )
 = \big( e^{2 \pi i \theta_1} \zt_1, \, e^{2 \pi i \theta_2} \zt_2, \,
     \ldots, \, e^{2 \pi i \theta_n} \zt_n \big)
\]
for $\zt_1, \zt_2, \ldots, \zt_n \in \T$.

Let $D = M_{\Q}$ be the universal UHF~algebra.
As in the proof of Theorem~\ref{T_5X17_AHEx},
let ${\overline{E}}^*_D (-)$ be the (reduced) cohomology theory
which arises
as in \cite[Corollary 3.9]{paper:DadarlatPennigSpectra}
from the infinite loop structure
of the classifying space of $\Aut_0 (D \otimes K)$.
By statement~(ii) at the beginning of the proof of
\cite[Corollary 4.5]{paper:DadarlatPennigSpectra},
${\overline{E}}^1_D (X) \cong \bigoplus_{k \geq1} H^{2 k + 1} (X; \Q)$.
Let $F$ be a locally trivial continuous field
of \ca{s} over $X$
with fibers isomorphic to $M_{\Q} \otimes K$
and structure group $\mathrm{Aut}_0 (M_{\Q} \otimes K)$.
As in \cite[Corollary 3.9]{paper:DadarlatPennigSpectra},
$F$ is determined up to isomorphism of bundles
by its
class in $[F] \in {\overline{E}}^1_D (X)$:
\[
[F] = (\delta_1 (F), \delta_2 (F), \delta_3 (F), ...)
 \in H^3 (X; \Q) \oplus H^5 (X; \Q) \oplus H^7 (X; \Q) \oplus \cdots.
\]
By \cite[Theorem 3.4]{paper:DadarlatPennigBrauer},
the opposite bundle $F^{\op}$ satisfies
$\delta_k (F^{\op}) = (-1)^k \delta_k (F)$ for $k \in \N$.
Therefore the class of $F^{\op}$ is given by
\[
[F^{\op}]
 = \big(- \delta_1 (F), \, \delta_2 (F), \, - \delta_3 (F),
   \, \ldots \big).
\]

Let $\xi \in H^1 (\T; \Q)$ be the standard generator.
For $k = 1, 2, \ldots, n$
let $p_k \colon X \to \T$ be the projection on the $k$-th coordinate,
and define $\xi_k = p_k^* (\xi) \in H^1 (X; \Q)$.
It is known that
$H^{*} (X; \Q) \cong \bigwedge^{*} (\Q^n)$ as graded rings,
with $\xi_1, \xi_2, \ldots, \xi_n$ forming a basis of $H^1 (X; \Q)$.
Define $\et_1 \in H^3 (X; \Q)$, $\et_2 \in H^5 (X; \Q)$,
and $\et_3 \in H^7 (X; \Q)$
to be the cup products
\[
\et_1 = \xi_1 \smile \xi_2 \smile \xi_3,
\quad
\et_2 = \xi_4 \smile \xi_5 \smile \cdots \smile \xi_8,
\]
and
\[
\et_3 = \xi_9 \smile \xi_{10} \smile \cdots \smile \xi_{15}.
\]
Then $\et_3 \smile \et_5 \smile \et_7 \neq 0$.
Using the correspondence above,
choose a locally trivial continuous field $E$ over $X$
with fiber $M_{\Q} \otimes K$
such that
\[
\dt_1 (E) = \et_1,
\quad
\dt_2 (E) = \et_2,
\quad
\dt_3 (E) = \et_3,
\quad
\dt_{7} (E) = \et_1 \smile \et_2 \smile \et_3,
\]
and $\dt_k (E) = 0$ for all other values of~$k$.
Then
\[
\dt_1 (E^{\op}) = - \et_1,
\quad
\dt_2 (E^{\op}) = \et_2,
\quad
\dt_3 (E^{\op}) = - \et_3,
\quad
\dt_{7} (E^{\op}) = - \et_1 \smile \et_2 \smile \et_3,
\]
and $\dt_k (E^{\op}) = 0$ for all other values of~$k$.

Suppose $\Gm (E^{\op}) \cong \Gm (E)$.
Then,
by reasoning analogous to that in the proofs of
Theorem~\ref{T_5X17_AHEx} and Theorem~\ref{T_5X17_KbgEx},
there must be a homeomorphism
$g \colon X \to X$
such that
$g^* (\dt_k (E^{\op})) = \dt_k (E)$
for all $k \in \N$.
But $g^*$ is a morphism of graded rings
$g^* \colon H^* (X; \Q) \to H^* (X; \Q)$.
Thus,
if
\[
g^* (- \et_1) = \et_1,
\quad
g^* (\et_2) = \et_1,
\quad
{\mbox{and}}
\quad
g^* (- \et_3) = \et_3,
\]
then
\[
g^* ( - \et_1 \smile \et_2 \smile \et_3 )
 = - \et_1 \smile \et_2 \smile \et_3
 \neq \et_1 \smile \et_2 \smile \et_3.
\]
So $\Gm (E^{\op}) \not\cong \Gm (E)$.

Presumably $\Gm (E)$ has no tracial states.
If we want to use $\Gm (E)$ in place of $A_0$
in the proof of Theorem~\ref{T_5X17_AHEx},
we need nonunital analogs of the theorems cited in that proof,
many of which are not known.

One may also use the space $X = S^3 \times S^5 \times S^7$,
taking $\et_1 \in H^3 (X; \Q)$, $\et_2 \in H^5 (X; \Q)$,
and $\et_3 \in H^7 (X; \Q)$
to be the classes coming from generators of
$H^3 (S^3; \Q)$, $H^5 (S^5; \Q)$,
and $H^7 (S^7; \Q)$
except that
for the existence of minimal homeomorphisms
one appeals to \cite{FH}
as in the proof of
Theorem~\ref{T_5X17_AHEx} and Theorem~\ref{T_5X17_KbgEx}.
\end{rmk}

We conclude with a few natural questions,
which we have not seriously investigated.

\begin{qst}\label{Q_5Y07_opKK}
\
\begin{enumerate}
\item\label{Q_5Y07_opKK_Ours}
Are the actions in
Theorem~\ref{T_5X17_AHEx} and Theorem~\ref{T_5X17_KbgEx}
$KK^{\T}$-equivalent to their opposite actions?
\item\label{Q_5Y07_opKK_Any}
Is there any circle action
on an algebra
as in Theorem~\ref{T_5X17_AHEx}
or Theorem~\ref{T_5X17_KbgEx}
which is not $KK^{\T}$-equivalent to
its opposite action?
\item\label{Q_5Y07_opKK_BtMy}
What happens to the Bentmann-Meyer
invariant (\cite{bentmann-meyer}) of an action of~$\T$
when one passes
to the opposite algebra
but keeps the same formula for the action?
\end{enumerate}
\end{qst}

\begin{qst}\label{Q_5Y07_Discrete}
What happens when we restrict the actions
of Theorem~\ref{T_5X17_AHEx} and Theorem~\ref{T_5X17_KbgEx}
to finite subgroups of~$\T$?
What happens if we consider these actions
as actions of $\T$ but with the discrete topology?
\end{qst}
\bibliographystyle{abbrv}


\end{document}